\documentclass[a4paper,11 pt]{amsart}
\usepackage{amsfonts}
\usepackage{amssymb}
\usepackage[utf8]{inputenc}
\usepackage{amsmath}
\usepackage{pdflscape}
\usepackage{graphicx}
\usepackage[colorlinks=true, linkcolor=red, citecolor=blue]{hyperref}
\usepackage[]{epsfig}
\usepackage[]{pstricks}
\usepackage{tikz}

\newcommand{\R}{{\mathbb R}}

\newtheorem{theorem}{Theorem}[section]

\newtheorem{lemma}[theorem]{Lemma}

\newtheorem{proposition}[theorem]{Proposition}

\theoremstyle{definition}

\newtheorem{remark}[theorem]{Remark}

\theoremstyle{remark}

\makeatletter
\@namedef{subjclassname@2020}{\textup{2020} Mathematics Subject Classification}
\makeatother

\newcommand{\field}[1]{\mathbb{#1}}
\newcommand{\C}{\field{C}}
\newcommand{\NN}{\mathcal{N}}
\newcommand{\FF}{\mathcal{F}}

\newcommand{\remove}[1]{ }

\def\R{\mathbb R}
\def\be{\begin{equation}}
\def\ee{\end{equation}}
\def\ba{\begin{eqnarray}}
\def\ea{\end{eqnarray}}

\numberwithin{equation}{section}

\begin{document}
\title[Kawahara equation with time-delayed boundary damping]
{\bf  Two stability results for the Kawahara equation with a time-delayed boundary control}

\author[Capistrano--Filho]{Roberto de A. Capistrano--Filho}
\address{Departamento de Matem\'atica,  Universidade Federal de Pernambuco (UFPE), 50740-545, Recife (PE), Brazil.}
\email{roberto.capistranofilho@ufpe.br}
\email{luan.soares@ufpe.br}
\email{victor.hugo.gonzalez.martinez@gmail.com}

\author[Chentouf]{Boumedi\`ene Chentouf*}
\address{Kuwait University, Faculty of Science, Department of Mathematics, Safat 13060, Kuwait}
\email{boumediene.chentouf@ku.edu.kw}
\thanks{*Corresponding author.}

\author[de Sousa ]{Luan S.  de Sousa}

\author[Gonzalez Martinez]{Victor H. Gonzalez Martinez}

\begin{abstract}
In this paper, we consider the Kawahara equation in a bounded interval and with a delay term in one of the boundary conditions.  Using two different approaches, we prove that this system is exponentially stable under a condition on the length of the spatial domain.  Specifically, the first result is obtained by introducing a suitable energy and using the Lyapunov approach, to ensure that the unique solution of the Kawahara system exponentially decays.  The second result is achieved by means of a compactness-uniqueness argument,  which reduces our study to prove an observability inequality.  Furthermore, the main novelty of this work is to characterize the critical set phenomenon for this equation by showing that the stability results hold whenever the spatial length is related to the Möbius transformations.
\end{abstract}

\date{\today}
\maketitle


\thispagestyle{empty}

 \normalsize

\section{Introduction}
\setcounter{equation}{0}

\subsection{Physical motivation and goal} It is well-known that the following fifth order nonlinear dispersive equation
\begin{equation}\label{kaw}
 \pm2 \partial_tu + 3u\partial_xu - \nu \partial^3_xu +\frac{1}{45}\partial^5_xu = 0,
\end{equation}
models numerous physical phenomena.  In fact, considering suitable assumptions on the amplitude, wavelength, wave steepness and so on, the properties of the asymptotic models for water waves  have been extensively studied in the last years, through \eqref{kaw}, to understand the full water wave system \footnote{ See for instance  \cite{ASL, BLS, Lannes} and references therein, for a rigorous justification of various asymptotic models for surface and internal waves.}.

In some situations, we can formulate the waves as a free boundary problem of the \linebreak incompressible, irrotational Euler equation in an appropriate non-dimensional form with at least two (non-dimensional) parameters $\delta := \frac{h}{\lambda}$ and $\varepsilon := \frac{a}{h}$, where the water depth, the wavelength and the amplitude of the free surface are parameterized as $h, \lambda$ and $a$, respectively.  In turn, if we introduce another non-dimensional parameter $\mu$, so-called the Bond number, which measures the importance of gravitational forces compared to surface tension forces, then the physical condition $\delta \ll 1$ characterizes the waves, which are called long waves or shallow water waves.

On the other hand, there are several long wave approximations depending on the relations between $\varepsilon$ and $\delta$.  For instance, if we consider $\varepsilon = \delta^4 \ll 1$ and $\mu = \frac13 + \nu\varepsilon^{\frac12},$ and in connection with the  critical Bond number $\mu = \frac13$,  we have the so-called Kawahara equation, represented by \eqref{kaw}, and derived by Hasimoto and Kawahara in \cite{Hasimoto1970,Kawahara}.

The main concern of this paper is to deal with the Kawahara equation in a bounded domain under the action of a time-delayed boundary control, namely
\begin{equation}\label{sis1}
 \left\{
 \begin{array}{ll}
\partial_{t} u(t,x)+a \partial_x u(t,x) +b\partial_x^3 u(t,x)- \partial_x^5 u(t,x) + u(t,x)^{p} \partial_x u(t,x)=0, & (t,x) \in  \mathbb{R}^{+}\times \Omega,\\
u (t,0) =u (t,L) =\partial_x u(t,0)=\partial_x u(t,L) =0, & t >0,\\
\partial_x^2 u(t,L)=\mathcal{F}(t,h), & t >0,\\
\partial_x^{2}u(t,0)=z_0(t), &  \; t \in  \mathcal{T}, \\
u(0,x) =u_{0} (x), & x \in  \Omega.
\end{array}
\right.
\end{equation}
In \eqref{sis1}, $\Omega=(0,L)$, where $L > 0$,  while $a >0$ and $b>0$ are physical parameters. Moreover, $p \in [1,2]$ and  $\mathcal{F}(t,h)$ is the delayed control given by
\begin{equation}\label{F}
\mathcal{F}(t)=\alpha \partial_x^2 u(t,0)+\beta \partial_x^2 u(t-h,0),
\end{equation}
in which $h > 0$ is the time-delay,   $\alpha$ and $ \beta$ are two feedback gains satisfying the restriction
\begin{equation}\label{ab}
|\alpha| + |\beta| <1.
\end{equation}
Finally, $\mathcal{T}=(-h,0)$, while $u_0$ and $z_0$ are initial conditions.

Thereafter, the functional energy associated to the system \eqref{sis1}-\eqref{F} is
\begin{equation}\label{energia}
    E(t)= \int_{0}^{L}u^2 (t,x)dx + h|\beta|\int_{0}^{1}(\partial_x^{2}u(t -h\rho,0))^{2}d\rho, \ t \geq 0.
\end{equation}
Now, recall that if $\alpha=\beta=0$, then the term $\partial_x^2 u(t,0)$ represents a feedback damping mechanism (see for instance \cite{ara,vasi2}) but an extra internal damping is required to achieve the stability of the solutions. Therefore, taking into account the action of the time-delayed boundary control \eqref{F} in \eqref{sis1}, the following issue will be dealt with in this article:

\vspace{0.1cm}

\textit{Does $E(t)\longrightarrow0$, as $t\to\infty$? If it is the case, can we provide a decay rate?}

\subsection{Historical background} Let us first present a review of the main  results available in the literature for the analysis of the Kawahara equation in a bounded interval. A pioneer work is due to Silva and Vasconcellos \cite{vasi1,vasi2}, where the authors studied the stabilization of global solutions of the linear Kawahara equation in a bounded interval under the effect of a localized damping mechanism.  The second endeavor, in this line, is completed by Capistrano-Filho \textit{et al.} \cite{ara}, where the generalized Kawahara equation in a bounded domain $Q_T = (0, T) \times (0,L)$ is considered
\begin{equation}\label{int1}
\left\lbrace
\begin{array}{llr}
\partial_t u + \partial_x u +\partial^3_x u - \partial^5_x u+u^p \partial_x u +a(x)u= 0, & \mbox{in} \ Q_{T}, \\
u(t,0)=u(t,L)=  \partial_x u(t,0)=\partial_x u (t,L)=\partial^2_xu (t,L) = 0, & \mbox{on} \ [0,T],\\
 u(0,x) = u_{0}(x), & \mbox{in} \ [0,L],
\end{array}\right.
\end{equation}
with $p\in [1,4)$. It is proven that the  solutions of the above system decays exponentially.

The internal controllability problem has been tackled by Chen \cite{MoChen} for the Kawahara equation with homogeneous boundary conditions. Using Carleman estimates associated to the linear operator of the Kawahara equation with an internal observation, a null controllable result is shown when the internal control is effective in a subdomain $\omega\subset(0,L)$.  In \cite{CaGo},  considering the system \eqref{int1} with an internal control $f(t,x)$ and homogeneous boundary conditions, the equation is exact shown to be controllable in $L^2$-weighted Sobolev spaces and, additionally,  controllable by regions in $L^2$-Sobolev space.

Recently,  a new tool for the control properties for the Kawahara operator was proposed in \cite{CaSo, CaSoGa}.  First, in \cite{CaSo}, the authors showed a new type of controllability for the Kawahara equation, what they called \textit{overdetermination control problem}. A boundary control is designed so that the solution of the problem under consideration satisfies an integral condition. Furthermore, when the control acts internally in the system, instead of the boundary, the authors proved that this integral condition is also satisfied.  After that, in \cite{CaSoGa}, the authors extend this idea to the internal control problem for the Kawahara equation on unbounded domains.  Precisely, under certain hypotheses over the initial and boundary data,  an internal control input is designed so that the solutions of the Kawahara equation satisfies an integral overdetermination condition, whether the Kawahara equation is posed in the real line, left half-line or right half-line.  We also note that the existence and uniqueness of solutions as well their stability are investigated for the Kawahara type equation posed in the whole real line \cite{coc1,coc2,coc3,cui,isa}, the half-line \cite{dor1, lar2}, a periodic domain \cite{hir, kat}, and a non-periodic bounded domain \cite{dor2, dor3, lar1, lar2}.

We conclude the literature review by mentioning the last works on the stabilization of the Kawahara equation with a localized time-delayed interior control. In \cite{CaVi, boumediene}, under suitable assumptions on the time delay coefficients, the authors are able to prove that solutions of the Kawahara system are exponentially stable. The results are obtained using Lyapunov approach and a compactness-uniqueness argument.

\subsection{Notations and main results} First of all, let us introduce the following notations that we will use throughout this manuscript.
\begin{itemize}
	\item[(i)] We consider the space of solutions
	\begin{equation*}
		X(Q_{T})= C(0,T;L^{2}(0,L))\cap L^{2}(0,T; H^{2}(0,L))
	\end{equation*}
equipped with the norm
\begin{equation*}
\|v\|_{X(Q_{T})}= \max_{t \in (0,T)}\|v(t,\cdot )\|_{L^{2}(0,L)} + \biggl(\int_{0}^{T}\|v(t,\cdot )\|^{2}_{H^{2}(0,L)}dt\biggr)^{\frac{1}{2}}.
\end{equation*}

	\item[(ii)] Denote by $$\tilde{H}= L^{2}(0,L)\times
	 L^{2}(-h,0)$$ the Hilbert space equipped with the inner product
	 \begin{equation*}
 \langle (u_1,z_1),(u_2,z_2)\rangle_{\tilde{H}}= \int_{0}^{L}u_{1}u_{2}dx + |\beta|\int_{-h}^{0}z_{1}(s)z_{2}(s)~ds,
	 \end{equation*}
	which yields the following norm
	$$\|(u,z)\|^{2}_{\tilde{H}} = \int_{0}^{L}u^2(x) dx + |\beta|\int_{-h}^{0}z^2(\rho) d\rho.$$


\item[(iii)] Throughout all the manuscript, $(\cdot,\cdot)_{\mathbb{R}^{2}}$ denotes the canonical inner product of $\mathbb{R}^{2}$.
\end{itemize}
With the above notations in hand,  let us state our first main result in this article:

\begin{theorem}\label{Lyapunov}Let $\alpha\neq0$ and $\beta\neq0$ be two real constants satisfying \eqref{ab} and suppose that the spatial length $L$ fulfills
\begin{equation}\label{L}
0< L < \sqrt{\frac{3b}{a}}\pi.
\end{equation}
Then,  there exists $r>0$ sufficiently small, such that for every $(u_{0}, z_{0})\in H$ with $\|(u_{0}, z_{0})\|_{H} < r$, the energy of the system \eqref{sis1}-\eqref{F}, denoted by $E$ and defined by \eqref{energia} exponentially decays, that is,  there exist two positive constants $\kappa$ and $\lambda$ such that
\begin{equation}\label{exp decay}
	E(t) \leq \kappa E(0)e^{-2\lambda t}, \ t > 0.
\end{equation}
Here,
\begin{equation}\label{lambda}
	\lambda \leq \min\left\{ \frac{\mu_{2}}{2h(\mu_{2} + |\beta|)} ,\frac{3b\pi^2-r^2L-L^2a}{2L^2(1+L\mu_1)}\mu_1\right\}
\end{equation}
and
\begin{equation*}
	\kappa \leq \biggl(1 + \max\left\{L\mu_{1}, \frac{\mu_{2}}{|\beta|}\right\}\biggr),
\end{equation*}
for $\mu_1,\mu_2\in(0,1)$ sufficiently small.
\end{theorem}

The second main result gives a second answer for the question presented in this introduction. Indeed, using a different approach based on an observability inequality, we are able to highlight the  critical lengths phenomenon observed in \cite{ara} for the Kawahara equation:

\begin{theorem}\label{main2}
Assume that  $\alpha$ and $\beta$ satisfy \eqref{ab}, whereas $L>0$ is taken so that the problem $(\NN)$ (see Lemma \ref{lem2}) has only the trivial solution. Then, there exists $r>0$ such that for every $\left(u_{0}, z_{0}\right) \in H$ satisfying
$$
\left\|\left(u_{0}, z_{0}\right)\right\|_{H} \leq r,
$$
the energy of system \eqref{sis1}-\eqref{F}, denoted by $E$ and defined by \eqref{energia}, decays exponentially.  More precisely, there exist two positive constants $\nu$ and $\kappa$ such that
$$
E(t) \leq \kappa E(0) e^{-\nu t}, \quad t>0.
$$
\end{theorem}

\subsection{Heuristic of the article and its structure}
In this article, we are able to prove that the Kawahara system \eqref{sis1} is exponentially stable when we consider a  boundary time-delayed control $\mathcal{F}(t)$ defined by \eqref{F}.

In order to prove Theorem \ref{Lyapunov}, we use the idea of the work that treated the delayed wave systems \cite{xyl} (see also \cite{np}).  More precisely, choosing an appropriate Lyapunov functional associated to the solutions of \eqref{sis1}-\eqref{F} and with some restrictions on the spatial length $L$ and an appropriate size of the initial data, that is,  $L$ bounded as in \eqref{L} and  $$\|(u_{0}, z_{0})\|_{H} < \frac{2}{\pi}\sqrt{\frac{3b\pi^2-L^2 a}{L}},$$ the energy \eqref{energia} decays exponentially.   The key idea of this analysis is the relation between the linearized system associated with \eqref{sis1}-\eqref{F} and a transport equation (see the Section \ref{Sec2} for more details).

With regard to the proof of Theorem \ref{main2}, we proceed as in \cite{ro}, i.e, combining multipliers and compactness arguments which reduces the problem to show a unique continuation result for the state operator. To prove the latter, we extend the solution under consideration by zero in $\mathbb{R}\setminus[0,L]$ and take the Fourier transform. However, due to the complexity of the system,  after taking the Fourier transform of the extended solution $u$,  it is not possible to use the same techniques used in \cite{ro}. Thus, to prove our main result we invoke the result due Santos \textit{et al.} \cite{santos}. Specifically, after taking the Fourier transform, the issue is to establish when a certain quotient of entire functions still turns out to be an entire function. We then pick a polynomial $q: \mathbb{C}\to\mathbb{C}$ and a family of functions
\begin{equation}\label{N}
N_{\alpha}: \mathbb{C}\times(0,\infty)\to\mathbb{C},
\end{equation}
with $\alpha\in\mathbb{C}^4\setminus \{0\}$, whose restriction $N_{\alpha}(\cdot,L)$ is entire for each $L>0$. Next, we consider a family of functions $f_{\alpha}(\cdot,L)$, defined by
\begin{equation}\label{f}
f_{\alpha}(\mu,L)=\frac{N_{\alpha}(\mu,L)}{q(\mu)},
\end{equation}
in its maximal domain. The problem is then reduced to determine $L>0$ for which there exists $\alpha\in\mathbb{C}^4\setminus \{0\}$ such that $f_{\alpha}(\cdot,L)$ is entire. In contrast with the analysis developed in \cite{ro}, this approach does not provide an explicit characterization of a critical set, if it exists, but only ensures that the roots of $f$ have a relations with the M\"obius transform (see the proof of Lemma \ref{lem2} above).

\vspace{0.1cm}

Finally, let us present the outline of our work: First,  in the Section \ref{Sec2}, we prove regularity properties of the solutions to the linear system associated with \eqref{sis1}-\eqref{F} and then show that the well-posedness of the problem \eqref{sis1}-\eqref{F}. Section \ref{Sec3} is devoted to the proof of the first main result of this article, Theorem \ref{Lyapunov}.  In the Section \ref{Sec4},  with the help of the result established in \cite{santos},  we  show Theorem \ref{main2}.  Finally,  in the Section \ref{Sec5}, we present some additional comments and open questions.

\section{Well-posedness results}\label{Sec2}
The goal of this section is to prove that the full nonlinear Kawahara system \eqref{sis1}-\eqref{F} is well-posed. The proof is divided into four parts by using the strategy due to Rosier \cite{ro}:
\begin{enumerate}
\item Well-posedness to the linear system associated to \eqref{sis1}-\eqref{F};
\item Properties of regularity of the linear system associated to \eqref{sis1}-\eqref{F}.
\item Well-posedness of the linear system associated to \eqref{sis1}-\eqref{F} with a source term.
    \item Well-posedness of the system \eqref{sis1}-\eqref{F}.
\end{enumerate}

\subsection{Well-posedness: Linear system}
We begin by proving the well posedness of the linearized system
 \begin{equation}\label{2.1}
	\left\{
	\begin{array}{ll}
		\partial_{t} u(t,x)+a \partial_x u(t,x) +b\partial_x^3 u(t,x)- \partial_x^5 u(t,x) =0, & (t,x) \in  \mathbb{R}^{+}\times\Omega,\\
		u (t,0) =u (t,L) =\partial_x u(t,0)=\partial_x u(t,L) =0, & t >0,\\
		\partial_x^2 u(t,L)=\alpha \partial_x^2 u(t,0)+\beta \partial_x^2 u(t-h,0), & t >0,\\
		u(0,x) =u_{0} (x) , & x \in  \Omega.\\
	\end{array}
	\right.
\end{equation}
In order to investigate \eqref{2.1}, let $z(t,\rho)=  \partial_x^2u(t - \rho h,0)$, which satisfies the transport equation \cite{xyl} (see also \cite{np})
\begin{equation}\label{trans. eq}
\left\{
\begin{array}{ll}
		h\partial_{t}z(t,\rho) + \partial_{\rho}z(t,\rho)= 0, & \rho \in (0,1), \ t \ > 0, \\
		z(t,0)= \partial_x^{2}u(t,0), & t >0,\\
		z(0,\rho)= z_{0}(-h\rho), & \rho \in (0,1).
	\end{array}
\right.	
\end{equation}
Next, we consider the Hilbert space $H=L^2(0,L) \times L^2(0,1)$ equipped with the following inner product
\begin{equation*}
\langle (u_{1},z_{1}),(u_{2}, z_{2})\rangle_H = \int_{0}^{L}u_{1}u_{2}dx + |\beta|h\int_{0}^{1}z_{1}z_{2}d\rho.
\end{equation*}
Subsequently,  one can rewrite \eqref{2.1}-\eqref{trans. eq} as follows
\begin{equation}
\begin{cases}
U_t(t)=AU(t), \quad t>0,\\
U(0)=U_0\in H,
\end{cases}
\end{equation}
where
\begin{equation*}
	 A= \left[\begin{array}{lll}
		-a\partial_x - b\partial_x^3 + \partial_x^5 & 0 \\
		0 & -\frac{1}{h}\partial_{\rho}
	\end{array}\right], \\  U(t)= \left[\begin{array}{ll}
		u(t,\cdot) \\
	    z(t,\cdot)
	\end{array}\right],	U_{0}= \left[\begin{array}{ll}
	u_{0}(\cdot) \\
	z_{0}(-h(\cdot))
\end{array}\right]
\end{equation*}
and
\begin{equation*}
\begin{split}
	D(A)= \{(u,z) \in H^{5}(0,L)\times H^{1}(0,1); u(0)=u(L)=\partial_x u(0)=\partial_x u(L)=0,\\	
\quad \partial_x^{2}u(0)= z(0), \partial_x^{2}u(L)= \alpha\partial_x^{2}u(0) + \beta z(1)\}.
	\end{split}
\end{equation*}
The next result ensures the well-posedness for the problem \eqref{2.1}.

\begin{proposition}\label{linear}
Assume that the constants $\alpha$ and $\beta$ satisfy \eqref{ab} and that $U_0\in H. $ Then, there exists a unique mild solution $U\in C([0,+\infty),H)$ for the system \eqref{2.1}. Additionally, considering $U_0\in D(A)$, we have a classical solution with the following regularity
$$U\in C([+\infty),D(A))\cap C^1([0,+\infty),H).$$
\end{proposition}
\begin{proof}
As the proof uses standard arguments, only a sketch of it will be provided.  Let $U=(u,z)\in D(A).$ Then, integrating by parts and using the boundary conditions of   \eqref{2.1} and \eqref{trans. eq}, we obtain
\begin{equation}
\begin{split}
	\langle AU(t), U(t)\rangle_H=& \frac{1}{2}\left( \alpha^{2}(\partial_x^{2}u(t,0))^{2} + 2\alpha\beta\partial_x^{2}u(t,0)\partial_x^{2}u(t - h,0) \right)\\
		&+ \frac{1}{2}\left(	\beta^{2}(\partial_x^{2}u(t - h,0))^{2} -(\partial_x^{2}u(t,0))^{2}  \right)\\
		 &+\frac{1}{2}\left(-|\beta|(\partial_x^{2}u(t - h,0))^{2}  +|\beta|(\partial_x^{2}u(t,0))^{2}\right)=\frac{1}{2}(M\eta(t), \eta(t))_{\mathbb{R}^{2}},
	\end{split}
\label{n1}
\end{equation}
where
\begin{equation}\label{Matrix A}
\eta=\left[\begin{array}{ll}
		\partial_x^{2}u(t,0) \\
	   \partial_x^{2}u(t - h,0))
	\end{array}\right]
	\ \mbox{and} \ M= \left[\begin{array}{ll}
		\alpha^{2} - 1 + |\beta| & \alpha\beta \\
		\alpha\beta & \beta^{2} - |\beta|
	\end{array}\right].
\end{equation}

Now,  observe that for the Adjoint of $A$,  denoted by $A^*$, is defined by
\begin{equation*}
	A^{*}= \left[\begin{array}{ll}
		a\partial_x + b\partial_x^3 - \partial_x^5 & 0 \\
		0 & \frac{1}{h}\partial_{\rho}
	\end{array}\right]
\end{equation*}
with
\begin{equation*}
\begin{split}
		D(A^{*})=\{(\varphi,\psi) \in H^{5}(0,L)\times H^{1}(0,1): \varphi(0)=\varphi(L)=\partial_x \varphi(0)=\partial_x \varphi(L)=0, \\
		 \psi(1)=\frac{\beta}{|\beta|}\partial_x^{2}\varphi(L), \ \partial_x^{2}\varphi(0)= \alpha\partial_x^{2}\varphi(L) + |\beta|\psi(0)\}.
	\end{split}
\end{equation*}
Similarly, we have,  for $V= (\varphi, \psi) \in D(A^{*}),$  that
 \begin{equation*}
\begin{split}
 	\langle A^{*}V, V \rangle_H 	=& 	\frac{1}{2}\bigg[(\alpha^{2} - 1 + |\beta|^{2})\partial_x^{2}\varphi(L) + 2\alpha|\beta|\partial_x^{2}\varphi(L)\psi(0) +( |\beta|^{2} - |\beta|)\psi(0)^{2}\biggl]\\
 	= &\frac{1}{2}( M^{*}\eta^{*},\eta^{*})_{\mathbb{R}^{2}},
 	\end{split}
\label{n2}
 \end{equation*}
where
\begin{equation}\label{Matrix A*}
	\eta^{*}=\left[\begin{array}{ll}
\partial_x^{2}\varphi(L) \\
	  \psi(0)
	\end{array}\right]
	\ \mbox{and}
	 \ M^{*}= \left[\begin{array}{ll}
		\alpha^{2} - 1 + |\beta| & \alpha|\beta| \\
		\alpha|\beta| & \beta^{2} - |\beta|
	\end{array}\right].
\end{equation}
Now, let us check that $M$ and $M^*$ are negative definite. For this, we will use the following lemma:
\begin{lemma}\label{lemmand}
	Let $M=(m_{ij})_{i,j} \in \mathbb{M}_{2 \times 2}(\mathbb{R})$ be a symmetric matrix. If $m_{11}<0$ and $\det(M)>0$, then $M$ is negative definite.
\end{lemma}
\begin{proof}
	It is sufficient to note that for all $u=(x ~ y) \neq (0 ~ 0)$ we have
	\begin{align*}
		uMu^{\top}=&m_{11}x^2+2xym_{12}+m_{22}y^2
		=m_{11}\left(x+\frac{m_{12}}{m_{11}}y\right)^2+\left(\frac{m_{11}m_{22}-m_{12}^{2}}{m_{11}}\right)y^2<0,
	\end{align*}
	which completes the proof.
\end{proof}

Now, we are in position to finish the proof.  From \eqref{Matrix A}, \eqref{Matrix A*} and the condition \eqref{ab}, we see that $m_{11}=m^{*}_{11}=\alpha^2-1+|\beta|<0$ and $$\det M= \det M^{*}=|\beta|((|\beta| - 1)^{2} - \alpha^{2})>0,$$  where $M=(m_{ij})_{i,j \in \{1,2\}}$ and $M^{*}=(m^{*}_{i,j})_{i,j \in \{1,2\}}$. Therefore, by virtue of Lemma \ref{lemmand}, it follows that $M$ and $M^{*}$ are negative definite and hence both $A$ and $A^{*}$ are dissipative in view of \eqref{n1} and \eqref{n2}.

Finally,  since $A$ and $A^{*}$ are densely defined closed linear operators and both $A$ and $A^*$ are dissipative,  one can use semigroups theory of linear operators \cite{paz} to claim that $A$ is a generator of a $C_{0}$--semigroups of contractions on $H$, together with the statements of Proposition \ref{linear}.
\end{proof}

\begin{remark}
It is important to point out that considering  $\alpha= \beta= 0$ or $\alpha \neq 0$ and $\beta = 0$, the well posedness of \eqref{2.1} is easily obtained.  Indeed,  if  $\alpha= \beta= 0$, the result follows from \cite[Lemma 2.1]{ara}. In the case when $\alpha \neq 0$ and $\beta = 0$, we have $Au=-a\partial_{x}-b\partial_{x}^{3}u+\partial_{x}^{5}u$ with domain $$D(A)=\{u \in H^{5}(0,L): u(0)=u(L)=\partial_x u(0)=\partial_x u(L)=0,\partial_x^{2}u(L)= \alpha\partial_x^{2}u(0)\}.$$
It may be seen that $A^{*}v=a\partial_{x}v+b\partial_{x}^{3}v-\partial_{x}^{5}v$ with domain $$D(A^{*})=\{v \in H^{5}(0,L): v(0)=v(L)=\partial_x v(0)=\partial_x v(L)=0,\partial_x^{2}v(0)= \alpha\partial_x^{2}v(L)\}$$ and we easily verifies that $$(Au,u)_{L^2(0,L)}=\frac{(\alpha^2-1)}{2}(\partial_{x}^{2}u(0))^2 \quad \text{and} \quad (A^{*}v,v)_{L^2(0,L)}=\frac{(\alpha^2-1)}{2}(\partial_{x}^{2}v(L))^2,$$ so in this case, it is necessary to take $|\alpha| < 1$ in order to obtain the well posedness result.
\end{remark}

\subsection{Regularity estimates: Linear system}
In the sequel, let $\{S(t)\}_{t\geq0}$ be the semigroup of contractions associated with the operator $A$. We have some \textit{a priori} estimates and regularity estimates for the linear systems \eqref{2.1} and \eqref{trans. eq}.
\begin{proposition}\label{lienar1}
Suppose that \eqref{ab} holds. Then, the application
 \begin{equation}\label{SSSS}
	\begin{array}{lll}
		&\mathcal{S}:  H \longrightarrow X(Q_{T})\times C(0,T; L^{2}(0,1))\\
		&(u_{0},z_{0}(-h(\cdot))) \longmapsto S(\cdot)(u_{0},z_{0}(-h(\cdot)))
	\end{array}
\end{equation}
is well defined and continuous. Moreover, for every $(u_{0}(\cdot),z_{0}(-h(\cdot))) \in H,$ we have $$(\partial_x^{2}u(\cdot,0),z(\cdot,1)) \in L^{2}(0,T)\times L^{2}(0,T)$$ and the following estimates hold
\begin{equation}\label{2.14}
\|\partial_x^{2}u(\cdot,0)\|^{2}_{L^{2}(0,T)} + \|z(\cdot,1)\|^{2}_{L^{2}(0,T)} \leq C\bigl(\|u_{0}\|^{2}_{L^{2}(0,L)} + \|z_{0}(-h(\cdot))\|^{2}_{L^{2}(0,1)}\bigl),
\end{equation}
\begin{equation}\label{2.15}
 \|u_{0}\|^{2}_{L^{2}(0,L)} \leq \frac{1}{T}\|u\|^{2}_{L^{2}(0,T; L^{2}(0,L))} + \|\partial_x^{2}u(\cdot,0)\|^{2}_{L^{2}(0,T)},	
\end{equation}
and
\begin{equation}\label{2.16}
\|z_{0}(-h(\cdot))\|^{2}_{L^{2}(0,1)} \leq \|z(T,\cdot)\|^{2}_{L^{2}(0,1)} + \frac{1}{h}\|z(\cdot,1)\|^{2}_{L^{2}(0,T)}
\end{equation}
for some constant $C > 0$ that may depend of $a,b,\alpha, \beta, L, T$ and $h$.
\end{proposition}
\begin{proof}
We split the proof in several steps.

\vspace{0.1cm}
\noindent{\textbf{Step 1.}}  \textit{Main identities.}
\vspace{0.1cm}

For every $(u_{0}, z_{0}(-h(\cdot))) \in H,$ the semigroups theory gives that $$S(\cdot)(u_{0}, z_{0}(-h(\cdot))) \in C(0,T; H)$$ and due to the fact that $A$ generates a $C_0$-semigroup of contractions, we have that
\begin{equation}\label{VIP}
\|u(t)\|^{2}_{L^{2}(0,L)} + h|\beta|\|z(t)\|^{2}_{L^{2}(0,1)}	\leq \|u_{0}\|^{2}_{L^{2}(0,L)} + h|\beta|\|z_{0}(-h(\cdot))\|^{2}_{L^{2}(0,1)}, \forall t \in [0,T] .
\end{equation}
Now, let $\phi \in C^{\infty}([0,1]\times[0,T]), \ \psi \in C^{\infty}([0,L]\times[0,T])$ and $(u, z) \in D(A).$ Then, multiplying \eqref{trans. eq} by $\phi z$ and \eqref{2.1} by $\psi u$, using integrations by parts and the initial conditions, we have
\begin{equation}\label{Wz}
\begin{split}
	\int_{0}^{1}[\phi(T,\rho)z(T,\rho)^{2} - \phi(0,\rho)z_{0}(-h\rho)^{2}]d\rho &- 	\frac{1}{h}\int_{0}^{T}\int_{0}^{1}[h\partial_{t}\phi(t,\rho) + \partial_{\rho}\phi(t,\rho)]z(t,\rho)^{2}d\rho dt\\
	&+	\frac{1}{h}\int_{0}^{T}[\phi(t,1)z(t,1)^{2} - \phi(t,0)(\partial_x^{2}u(t,0))^{2}]dt= 0
\end{split}
\end{equation}	
and
\begin{equation}\label{gamma}
\begin{split}
&	-\int_{0}^{T}\int_{0}^{L}[\partial_{t}\psi(t,x) + a\partial_x\psi(t,x) + b\partial_x^{3}\psi(t,x) - \partial_x^{5}\psi(t,x)]u^2 (t,x)dxdt \\
	& +
	3b\int_{0}^{T}\int_{0}^{L}\partial_x\psi(t,x)(\partial_xu(t,x))^{2}dxdt + \int_{0}^{L}[\psi(T,x)u^2 (t,x) - \psi(0,x)u_{0}(x)^{2}]dx\\
&	+
	5\int_{0}^{T}\int_{0}^{L}[\partial_x\psi(t,x)(\partial_x^{2}u(t,x))^{2} - \partial_x^{3}\psi(t,x)(\partial_xu(t,x))^{2}]dxdt\\
&	 -\int_{0}^{T}\psi(t,L)[\alpha \partial_x^{2}u(t,0) + \beta z(t,1)]^{2}dt + \int_{0}^{T}\psi(t,0)(\partial_x^{2}u(t,0))^{2}dt= 0.
	\end{split}
\end{equation}

\vspace{0.1cm}
\noindent{\textbf{Step 2.}} \textit{Proof of \eqref{2.14}.}
\vspace{0.1cm}

Let us pick $\phi(t,\rho)= \rho$ in \eqref{Wz} to get
\begin{equation*}
	\int_{0}^{1}(z(T,\rho)^{2} - z_{0}(-\rho h)^{2})\rho d\rho -\frac{1}{h}\int_{0}^{T}\int_{0}^{1}z(t,\rho)^{2} d\rho dt + \frac{1}{h}\int_{0}^{T}z(t,1)^{2}dt= 0.
\end{equation*}
Owing to \eqref{VIP}, the latter gives
\begin{equation}\label{1111}
	\|z(\cdot,1)\|_{L^{2}(0,T)}^{2} \leq (T+1)\left(1+\frac{1}{h|\beta|} \right)\biggl(\|u_{0}\|_{L^{2}(0,L)}^{2} + \|z_{0}(-h(\cdot))\|_{L^{2}(0,1)}^{2}\biggl).
\end{equation}
Now,  choosing $\psi(t,x)= 1$ in \eqref{gamma}  yields that
\begin{equation*}
	\int_{0}^{L}[u^2 (t,x) - u_{0}(x)^{2}]dx + \int_{0}^{T}(\partial_x^{2}u(t,0))^{2}dt - \int_{0}^{T}[\alpha \partial_x^{2}u(t,0)^{2} + \beta z(t,1)]^{2}dx= 0,
\end{equation*}
which implies
\begin{equation}\label{key}
	\int_{0}^{T}(\partial_x^{2}u(t,0))^{2}dt \leq \int_{0}^{T}(\alpha\partial_x^{2}u(t,0) + \beta z(t,1))^{2}dt + \|u_{0}\|^{2}_{L^{2}(0,L)}.
\end{equation}
Since
\begin{equation}\label{JB}
	(\alpha\partial_x^{2}u(t,0) + \beta z(t,1))^{2} \leq (\alpha^{2} + \beta^{2})((\partial_x^{2}u(t,0))^2 + (z(t,1))^{2}),
\end{equation}
it follows from \eqref{key} and \eqref{JB} that
\begin{equation*}
\int_{0}^{T}\bigl(1-(\alpha^{2} + \beta^{2}) \bigr)(\partial_x^{2}u(t,0))^{2}dt \leq \int_{0}^{T}(\alpha^{2} + \beta^{2})z(t,1)^{2}dt + \|u_{0}\|^{2}_{L^{2}(0,L)}.	
\end{equation*}
In view of \eqref{1111} and \eqref{ab}, the last estimate yields
\begin{equation}\label{Atil}
	\|\partial_x^{2}u(\cdot,0)\|_{L^{2}(0,T)}^2 \leq (T+1)\frac{1}{1-(\alpha^2+\beta^2)}\left(1+\frac{1}{h|\beta|} \right)\biggl(\|u_{0}\|^{2}_{L^{2}(0,L)} + \|z(-h(\cdot)))\|^{2}_{L^{2}(0,1)}\biggr).
\end{equation}
Combining \eqref{Atil} and \eqref{1111}, the estimate \eqref{2.14} follows.

\vspace{0.1cm}
\noindent{\textbf{Step 3.}} \textit{The map \eqref{SSSS} is well-defined and continuous.}
\vspace{0.1cm}

Letting $\psi(t,x) = x$  in \eqref{gamma} gives
\begin{equation*}
	\begin{split}
		-a\int_{0}^{T}\int_{0}^{L}u^2 (t,x)dxdt + 3b\int_{0}^{T}\int_{0}^{L}(\partial_xu(t,x))^{2}dxdt+5\int_{0}^{T}\int_{0}^{L}(\partial_x^{2}u(t,x))^{2}dx dt\\+
		\int_{0}^{L}x[u^2 (t,x) - u_{0}(x)^{2}]dx - L\int_{0}^{T}[\alpha\partial_x^{2}u(t,0) + \beta z(t,1)]^{2}dt= 0.
	\end{split}
\end{equation*}
which implies, using \eqref{VIP} and \eqref{JB},  that
\begin{equation*}
	\begin{split} 3b\int_{0}^{T}\int_{0}^{L}(\partial_xu(t,x))^{2}dxdt+&5\int_{0}^{T}\int_{0}^{L}(\partial_x^{2}u(t,x))^{2}dx dt \leq a\bigl(\|u_{0}\|_{L^{2}(0,L)}^{2} +  h|\beta|\|z_{0}(-h(\cdot))\|^{2}_{L^{2}(0,1)}\bigr) \\
		&+  L\|u_{0}\|_{L^{2}(0,L)}^{2}+L(\alpha^{2} + \beta^{2})\bigr(\|\partial_x^{2}u(\cdot,0)\|^{2}_{L^{2}(0,T)} + \|z(\cdot,1)\|^{2}_{L^{2}(0,T)}\bigl).
	\end{split}
\end{equation*}
In light of \eqref{2.14}, we  deduce that
\begin{equation}\label{Aux}
\begin{split}
			\|\partial_xu\|_{L^{2}(0,T;L^{2}(0,L))}^{2}&+\|\partial_{xx} u\|_{L^2(0,T,L^2(0,L))}^{2} \leq  \frac{a}{\min\{3b,5\}}\bigl(\|u_{0}\|_{L^{2}(0,L)}^{2} +  h|\beta|\|z_{0}(-h(\cdot))\|^{2}_{L^{2}(0,1)}\bigr) \\
		&+(T+1)\frac{2-(\alpha^2+\beta^2)}{1-(\alpha^2+\beta^2)}\left(1+\frac{1}{h|\beta|} \right)\frac{L}{\min\{3b,5\}}(\alpha^{2} + \beta^{2})\\
		&\times \bigr(\|u_{0}\|^{2}_{L^{2}(0,L)} + \|z_{0}(-h(\cdot))\|^{2}_{L^{2}(0,1)}\bigl) +  \frac{L}{\min\{3b,5\}}\|u_{0}\|_{L^{2}(0,L)}^{2} \\
		\leq &C_{0}(T+1)\bigr(\|u_{0}\|^{2}_{L^{2}(0,L)} + \|z_{0}(-h(\cdot))\|^{2}_{L^{2}(0,1)}\bigl),
		\end{split}
	\end{equation}
where $$C_{0}= \max\left\{\frac{a}{\min\{3b,5\}},\frac{a}{\min\{3b,5\}}|\beta|h, \left(\frac{2-(\alpha^2+\beta^2)}{1-(\alpha^2+\beta^2)}\left(1+\frac{1}{h|\beta|} \right)\frac{L}{\min\{3b,5\}}(\alpha^{2} + \beta^{2})\right) \right\}.$$
Combining \eqref{Aux} and \eqref{VIP}, we obtain the desired result.

\vspace{0.1cm}
\noindent{\textbf{Step 4.}} \textit{Proof of \eqref{2.15} and \eqref{2.16}.}
\vspace{0.1cm}

In order to show these inequalities, choose $\psi= T - t $ in \eqref{gamma} and $\phi(t,\rho)= 1$ in \eqref{Wz}, respectively.  Performing similar computations as we did in step 2, the result follows. Moreover, owing to the density of $D(A)$ in $H$,  the proof of Proposition \ref{lienar1} is achieved.
\end{proof}

\subsection{Well-posedness: Linear system with a source term}
Now we consider the linear system with a source term
\begin{equation}\label{with source term}
	\left\{
	\begin{array}{ll}
		\partial_{t} u(t,x)+a \partial_x u(t,x) +b\partial_x^3 u(t,x)- \partial_x^5 u(t,x) =f(t,x), & (t,x) \in  \mathbb{R}^{+}\times\Omega,\\
		u (t,0) =u (t,L) =\partial_x u(t,0)=\partial_x u(t,L) =0, & t >0,\\
		\partial_x^2 u(t,L)=\alpha \partial_x^2 u(t,0)+\beta \partial_x^2 u(t-h,0), & t >0,\\
		\partial_x^2 u(t,0)=z_{0}(t), & t >0,\\
		u(0,x) =u_{0} (x) , & x \in  \Omega.\\
	\end{array}
	\right.
\end{equation}
Then, we have the following result.
\begin{proposition}\label{source term}
Let $|\alpha| $ and $|\beta|$ satisfying \eqref{ab}.  For every $(u_{0}, z_{0}) \in H$ and $f \in L^{2}(0,T;L^{2}(0,L)),$ there exists a unique mild solution $(u, \partial_x^{2}u(t- h.,0)) \in X(Q_{T})\times C(0,T; L^{2}(0,1))$ to \eqref{with source term}.  Moreover, there exists a constant $C > 0$ such that
\begin{equation}
	\|(u,z)\|_{C(0,T; H)} \leq C\biggl(\|(u_{0}, z_{0}(-h(\cdot)))\|_{H} + \|f\|_{L^{1}(0,T;L^{2}(0,L))}\biggr)
\end{equation}
and
\begin{equation}
	\|\partial_x^{2}u\|^{2}_{L^{2}(0,T; L^{2}(0,L))} \leq C\biggl(\|(u_{0}, z_{0}(-h(\cdot)))\|^{2}_{H} + \|f\|_{L^{1}(0,T;L^{2}(0,L))}^{2}\biggr).
\end{equation}
\end{proposition}
\begin{proof}
This proof is analogous to that of \cite[Proposition 2]{BCV} and hence we omit it.
\end{proof}

\subsection{Well-posedness of the nonlinear system \eqref{sis1}-\eqref{F}}
Let us now prove that the system \eqref{sis1}-\eqref{F} is well-posed.  To do so, we first deal with the properties of the nonlinearities, through the following lemma.

\begin{lemma}\label{uux} Let $u \in L^{2}(0,T; H^{2}(0,L))= L^{2}(H^{2})$. Then, $u\partial_xu$ and $u^{2}\partial_xu$ belong to $L^{1}(0,T;L^{2}(0,L))$. Besides, there exists positives constants $C_{0}$ and $ C_{1}$, depending of $L$, such that for every $u, v \in  L^{2}(0,T; H^{2}(0,L))$, one has
	\begin{equation}\label{nonlin 1}
		\int_{0}^{T}\|u_{1}\partial_xu_{1} -  u_{2}\partial_xu_{2}\|_{L^{2}(0,L)}dt \leq C_{0}(\|u_{1}\|_{L^{2}(H^{2})} + \|u_{2}\|_{L^{2}(H^{2})})\|u_{1} -  u_{2}\|_{L^{2}(H^{2})}
	\end{equation}
and
	\begin{equation}\label{nonlin 2}
	\int_{0}^{T}\|u^{2}_{1}\partial_xu_{1} -  u^{2}_{2}\partial_xu_{2}\|_{L^{2}(0,L)}dt \leq C_{0}(1+T^\frac{1}{2})\left( \|u\|_{X(Q_T)}^{2}+\|v\|_{X(Q_T)}^{2}\right) \|u-v\|_{X(Q_T)}.
\end{equation}
\end{lemma}
\begin{proof} Observe that \eqref{nonlin 1} follows from \cite[Lemma 2.1, p. 106]{vasi}.  Concerning \eqref{nonlin 2}, note that $$\sup_{x \in (0,L)}|u(x)^{2}| \leq \|u\|_{L^2(0,L)}^{2}+\|u\|_{L^2(0,L)}\|\partial_{x}u\|_{L^2(0,L)},$$
for $u \in H^1(0,L)$.  As we have that $ H^{2}(0,L) \hookrightarrow H^{1}(0,L) \hookrightarrow L^2(0,L)$, we get that
\begin{equation}\label{G3}
	z^2(x)-z^2(0)= \int_{0}^{x}(z^2(s))'ds 	=\int_{0}^{x}2z(s)z'(s)ds
	\leq 2 \|z\|_{L^2(0,L)}\|z'\|_{L^2(0,L)},
\end{equation}
for $u, v \in L^{2}(0,T; H^{2}(0,L))$. Consequently, in light of \eqref{G3}, we obtain
\begin{equation}\label{G4}
	\|z\|_{L^\infty(0,L)}^2 \leq 2 \|z\|_{L^2(0,L)}\|z'\|_{L^2(0,L)}.
\end{equation}
Let $u,z \in X(Q_T)$. We have
\begin{equation*}
\begin{split}
	\|u^2(\partial_{x}u-\partial_{x} v)\|_{L^1(0,T;L^2(0,L))}
	=& \int_{0}^{T}  \|u(t,\cdot)\|_{L^\infty(0,L)}^{2}\|(\partial_{x}u-\partial_{x} v)(t,\cdot)\|_{L^2(0,L)}dt\\
	\leq & \ T^{\frac{1}{2}}\|u\|_{L^\infty(0,T;L^2(0,L))}^{2} \|u-v\|_{L^2(0,T;H^2(0,L))}\\
	&+\|u\|_{L^\infty(0,T;L^2(0,L))}\|u\|_{L^2(0,T;H^2(0,L))}\|u-v\|_{L^2(0,T;H^2(0,L))}.
\end{split}
\end{equation*}
On the other hand, we have
\begin{equation*}
\begin{split}
\|(u^2-v^2) \partial_{x}v\|_{L^1(0,T;L^2(0,L))}=& \int_{0}^{T} \left(\int_{0}^{L}|u+v|^2|u-v|^2|\partial_{x}v|^2dx  \right)^{\frac{1}{2}}dt\\
\le& \int_{0}^{T} \left(\|(u+v)(t,\cdot)\|_{L^\infty(0,L)}^{2}\|(u-v)(t,\cdot)\|_{L^\infty(0,L)}^{2}\int_{0}^{L}|\partial_x v|^2dx  \right)^{\frac{1}{2}}dt\\
=& \int_{0}^{T} \|(u+v)(t,\cdot)\|_{L^\infty(0,L)}\|(u-v)(t,\cdot)\|_{L^\infty(0,L)} \|\partial_{x} v(t,\cdot)\|_{L^2(0,L)}dt.
\end{split}
\end{equation*}
Now, observe that
\begin{equation*}
\begin{split}
	&\|(u+v)(t,\cdot)\|_{L^\infty(0,L)}	\|(u-v)(t,\cdot)\|_{L^\infty(0,L)}\le\\
	&  \left(	\|(u+v)(t,\cdot)\|_{L^2(0,L)}+	\|(u+v)(t,\cdot)\|_{L^2(0,L)}^{\frac{1}{2}}\|(\partial_xu+\partial_x v)(t,\cdot)\|_{L^2(0,L)}^{\frac{1}{2}} \right)\\
	&\times \left(	\|(u-v)(t,\cdot)\|_{L^2(0,L)}+	\|(u-v)(t,\cdot)\|_{L^2(0,L)}^{\frac{1}{2}}\|(\partial_xu-\partial_x v)(t,\cdot)\|_{L^2(0,L)}^{\frac{1}{2}} \right)\\
	\leq & \|(u+v)(t,\cdot)\|_{L^2(0,L)}\|(u-v)(t,\cdot)\|_{L^2(0,L)}+\|(u+v)(t,\cdot)\|_{L^2(0,L)}\|(u-v)(t,\cdot)\|_{L^2(0,L)}\\
	&+\|(u+v)(t,\cdot)\|_{L^2(0,L)}\|(\partial_x u-\partial_x v)(t,\cdot)\|_{L^2(0,L)}+\|(u-v)(t,\cdot)\|_{L^2(0,L)}\|(u+v)(t,\cdot)\|_{L^2(0,L)}\\
	&+\|(u-v)(t,\cdot)\|_{L^2(0,L)}\|(\partial_x u+\partial_x v)(t,\cdot)\|_{L^2(0,L)}+\|(u+v)(t,\cdot)\|_{L^2(0,L)}\|(\partial_x u-\partial_x v)(t,\cdot)\|_{L^2(0,L)}\\
	&+\|(\partial_x u+\partial_x v)(t,\cdot)\|_{L^2(0,L)}\|(u-v)(t,\cdot)\|_{L^2(0,L)}.
\end{split}
\end{equation*}
Hence
\begin{equation*}
	\|u^2\partial_x u-v^2\partial_x v\|_{L^1(0,T;L^2(0,L))} \le (1+T^\frac{1}{2})\left( \|u\|_{X(Q_T)}^{2}+\|v\|_{X(Q_T)}^{2}\right) \|u-v\|_{X(Q_T)},
\end{equation*}
and thus \eqref{nonlin 2} is proved.
\end{proof}

Finally,  combining the previous lemma with the Proposition \ref{source term},  with a classical  fixed-point argument (see, for instance, \cite{ara}), we can obtain the following well-posedness result.

\begin{theorem}\label{well posedness}
Let $L > 0$, $a, b>0$ and $\alpha, \beta \in \mathbb{R}$ satisfying \eqref{ab}.  Assume $p \in [1,2]$ and $h > 0.$ If $u_{0} \in L^{2}(0,L)$ and $z_{0} \in L^{2}(0,1)$ are sufficient small, then the system \eqref{sis1}-\eqref{F} admits a unique solution $u \in X(Q_{T}).$
\end{theorem}

\section{A stabilization result \textit{via} Lyapunov approach} \label{Sec3}
The aim of this part of the work is to prove our first main result presented in Theorem \ref{Lyapunov}.  Precisely, we will prove the case $p=2$, that is, when the nonlinearity takes the form $u^{2}\partial_xu$. The case $u \partial_x u$ can be shown in a similar way, therefore, we will omit its proof.
\begin{proof}[Proof of Theorem \ref{Lyapunov}]
First, we choose the following Lyapunov functional
\begin{equation*}
	V(t)= E(t) + \mu_{1}V_{1}(t) + \mu_{2}V_{2}(t).
\end{equation*}
Here $\mu_{1}, \mu_{2} \in (0,1)$,  $V_1$ is defined by
\begin{equation}\label{V1}
	V_{1}(t)= \int_{0}^{L}xu^2 (t,x)dx
\end{equation}
and $V_2$ is defined by
 \begin{equation*}
 	V_{2}(t)= h\int_{0}^{1}(1 - \rho)(\partial_x^{2}u(t - h\rho,0))^{2}d\rho,
 \end{equation*}
for any regular solution of \eqref{sis1}-\eqref{F}.  Clearly, we have the following
\begin{equation}\label{E and V}
	E(t) \leq V(t),
\end{equation}
for all $t\geq0$.  On the other hand, we have
\begin{equation*}
\begin{split}
	\mu_{1}V_{1}(t) + \mu_{2}V_{2}(t)= &\mu_{1}\int_{0}^{L}xu^2 (t,x)dx + h\mu_{2}\int_{0}^{1}(1 - \rho)(\partial_x^{2}u(t - h\rho,0))^{2}d\rho\\
	\leq&	\mu_{1}L\int_{0}^{L}u^2 (t,x)dx + \mu_{2}\frac{h}{|\beta|}|\beta|\int_{0}^{1}(1 - \rho)(\partial_x^{2}u(t - h\rho,0))^{2}d\rho\\
	\leq &\max\left\{\mu_{1}L, \frac{\mu_{2}}{|\beta|}\right\}E(t), 	
\end{split}	
\end{equation*}
that is,
\begin{equation}\label{V and E}
	E(t)\leq V(t) \leq \biggr(1 + \max\left\{\mu_{1}L, \frac{\mu_{2}}{|\beta|}\right\}\biggr)E(t),
\end{equation}
for all $t\geq0$.

Now, consider a sufficiently regular solution  $u$  of \eqref{sis1}-\eqref{F}.  Differentiating $V_{1}(t)$, using integration by parts and  the boundary condition of \eqref{sis1}-\eqref{F}, it follows that
\begin{equation}\label{deriv of V1}
\begin{split}
		\frac{d}{dt}V_{1}(t)
		=&-2\int_{0}^{L}xu(t,x)\bigl[a\partial_x u +b\partial_x^{3}u - \partial_x^{5}u +u^{2}\partial_x u\bigr](t,x)dx\\
		=&a\int_{0}^{L} u^2 (t,x)dx - 3b\int_{0}^{L}(\partial_xu(t,x))^{2}dx -5\int_{0}^{L}(\partial_x^{2}u(t,x))^{2}dx + \frac{1}{2}\int_{0}^{L}u^4 (t,x)dx\\
&+ L\biggl[\alpha^{2}(\partial_x^{2}u(t,0))^{2} + 2\alpha\beta\partial_x^{2}u(t,0)\partial_x^{2}u(t - h,0) + \beta^{2}(\partial_x^{2}u(t - h,0))^{2}\biggr].
	\end{split}
\end{equation}
Similarly, in view of \eqref{trans. eq}, we have
\begin{equation}\label{deriv of V2}
\begin{split}
		\frac{d}{dt}V_{2}(t)=& 2h\int_{0}^{1}(1 - \rho)\partial_x^{2}u(t - \rho h,0)\frac{d}{dt}\partial_x^{2}u(t - \rho h,0)d\rho\\
		=&		\partial_x^{2}u(t,0)^{2} - \int_{0}^{1}(\partial_x^{2}u(t -\rho h,0))^{2}d\rho.
	\end{split}
\end{equation}
Consequently, \eqref{deriv of V1} and \eqref{deriv of V2} imply that for any $\lambda>0$
\begin{equation*}
\begin{split}
	\frac{d}{dt}V(t) + 2\lambda V(t)=&
	\biggl(\alpha^{2} - 1 + |\beta| + L\mu_{1}\alpha^{2} + \mu_{2}\biggl)(\partial_x^{2} u(t,0))^{2} + \biggl(\beta^{2} - |\beta| + L\mu_{1}\beta^{2}\biggr)(\partial_x^{2} u(t- h,0))^{2}\\
	&
	+ 2\alpha\beta\biggl(1 + L\mu_{1}\biggr)\partial_x^{2} u(t,0)\partial_x^{2} u(t- h,0) + (2\lambda h|\beta| - \mu_{2})\int_{0}^{1}(\partial_x^{2}u(t - \rho h,0))^{2}d\rho \\
	&
	+ 2\lambda\mu_{2}h\int_{0}^{1}(1-\rho)(\partial_x^{2}u(t - \rho h,0))^{2}d\rho + 2\lambda \mu_{1}\int_{0}^{L}x u^2 (t,x)dx + \frac{\mu_{1}}{2}\int_{0}^{L}u^4 (t,x)dx\\
	&
	+ (\mu_{1}a + 2\lambda)\int_{0}^{L}u^2 (t,x)dx -3b\mu_{1}\int_{0}^{L}(\partial_xu(t,x))^{2}dx - 5\mu_{1}\int_{0}^{L}(\partial_x^{2}u(t,x))^{2}dx,
	\end{split}
\end{equation*}
or equivalently,  by reorganizing the terms
\begin{equation}\label{Matrix comp.}
\begin{split}
		\frac{d}{dt}V(t) + 2\lambda V(t) \leq& \bigl(M_{\mu_{1}}^{\mu_{2}}~\eta(t), \eta(t)\bigr)_{\mathbb{R}^{2}}	- 3b\mu_{1}\int_{0}^{L}(\partial_x u(t,x))^{2}dx - 5\mu_{1}\int_{0}^{L}(\partial_x^{2} u(t,x))^{2}dx  \\
		&+ \bigl(2\lambda h(\mu_{2} + |\beta|) - \mu_{2}\bigr)\int_{0}^{1}(\partial_x^{2}u(t - \rho h,0))^{2}d\rho \\
		&+ \bigl(\mu_{1}a + 2\lambda(1 + L\mu_{1})\bigr)\int_{0}^{L}u^2 (t,x) dx+ \frac{\mu_{1}}{2}\int_{0}^{L}u^{4}(t,x)dx,
	\end{split}
\end{equation}
where $\eta(t)=(\partial_x^{2}u(t,0),\partial_x^{2}u(t-h,0))$ and
\begin{equation*}
	\begin{split}
	M_{\mu_{1}}^{\mu_{2}}= \left[\begin{array}{ll}
		(1 +L\mu_{1})\alpha^{2} - 1 + |\beta| + \mu_{2} & \alpha\beta(1 + L\mu_{1}) \\
		\alpha\beta(1 + L\mu_{1}) & \beta^{2} - |\beta| + L\mu_{1}\beta^{2}
	\end{array}\right].
\end{split}
\end{equation*}
Observe that
\begin{equation*}
	\begin{split}
	M_{\mu_{1}}^{\mu_{2}}=M+ L\mu_{1}\left[\begin{array}{ll}
	\alpha^{2} & \alpha\beta \\
	\alpha\beta & \beta^{2}
\end{array}\right]  + \mu_{2}\left[\begin{array}{ll}
1 & 0 \\
0 & 0
\end{array}\right],
\end{split}
\end{equation*}
where $M$ is defined by \eqref{Matrix A}.  Since $M$ is negative definite (see the proof of Proposition \ref{linear} and by virtue of the continuity of the determinant and the trace, one can claim that for $\mu_1$ and $\mu_2>0$ small enough, the matric $M_{\mu_{1}}^{\mu_{2}}$ can also be made negative definite.

Finally,  taking into account $\mu_1$ and $\mu_2>0$ are small enough and using Poincar\'e inequality\footnote{$\|u\|^{2}_{L^{2}(0,L)} \leq \frac{L^{2}}{\pi^{2}}\|\partial_xu\|_{L^{2}(0,L)}$ for $u \in H_{0}^{2}(0,L)$.}, we find
\begin{equation}\label{I}
	\begin{split}
		\frac{d}{dt}V(t) + 2\lambda V(t) \leq & \bigl(2\lambda h(\mu_{2} + |\beta|) - \mu_{2}\bigr)\int_{0}^{1}(\partial_x^{2}u(t - \rho h,0))^{2}d\rho \\
		& - 5\mu_{1}\int_{0}^{L}(\partial_x^{2} u(t,x))^{2}dx + \frac{\mu_{1}}{2}\int_{0}^{L}u^{4}(t,x) dx\\
&	+ \biggl(\frac{L^{2}}{\pi^{2}}(\mu_{1}a + 2\lambda(1 + L\mu_{1}))- 3b\mu_{1}\biggr)\int_{0}^{L}(\partial_xu^{2}(t,x))^{2}dx.
\end{split}
\end{equation}
Additionally, applying Cauchy-Schwarz inequality and using the facts that the energy $E$ defined by \eqref{energia} is nonincreasing, together with $H_{0}^{1}(0,L) \hookrightarrow L^{\infty}(0,L)$,   we have
\begin{equation}\label{II}
	\begin{split}
		 \frac{\mu_{1}}{2}\int_{0}^{L}u^{4}(t,x) dx \leq &\frac{\mu_{1}}{2}\|u(t,\cdot )\|^{2}_{L^{\infty}(0,L)}\int_{0}^{L}u^{2}(t,x)dx\\
		\leq &
		\frac{\mu_{1}}{2}L\|\partial_xu(t,\cdot )\|^{2}_{L^{2}(0,L)}\|u(t,x)\|^{2}_{L^{2}(0,L)}\\
		\leq &
		\frac{L \mu_{1}}{2}\bigl(\|u_{0}\|^{2}_{L^{2}(0,L)} + h|\beta|\|z_{0}(-h(\cdot))\|^{2}_{L^{2}(0,1)}\bigr)\|\partial_x u(t,\cdot )\|^{2}_{L^{2}(0,L)}\\
		\leq &
		\frac{L \mu_{1}}{2}\|(u_{0}, z_{0})\|^{2}_{H}\|\partial_x u(t,\cdot )\|^{2}_{L^{2}(0,L)}
		\leq
		r^{2}\frac{L \mu_{1}}{2} \|\partial_x u(t,\cdot )\|^{2}_{L^{2}(0,L)}.
\end{split}
\end{equation}
Combining \eqref{I} and \eqref{II} yields
\begin{equation} \label{Dt V}
		\frac{d}{dt}V(t) + 2\lambda V(t) \leq
		\Xi\|\partial_x^{2} u(t,x)\|_{L^{2}(0,L)}^{2}  + \bigl(2\lambda h(\mu_{2} + |\beta|) - \mu_{2}\bigr)\|\partial_x^{2}u(t -  \rho h,0)\|^{2}_{L^{2}(0,1)},
\end{equation}
where
$$\Xi= \frac{L \mu_{1}}{2}r^{2} + \frac{L^{2}}{\pi^{2}}\bigl(\mu_{1}a + 2\lambda(1 + L\mu_{1})\bigr) - 3b \mu_1.$$
In view of the constraint  \eqref{L} on the length $L$, one can choose $r$ small enough to get $$0<r<\frac{2}{\pi}\sqrt{\frac{3b\pi^2-L^2 a}{L}}.$$
Then, we pick $\lambda>0$ such that \eqref{lambda} holds to ensure that
\begin{equation}\label{Dt V2}
	\frac{d}{dt}V(t) + 2\lambda V(t) \leq 0,
\end{equation}
for all $t>0$. Therefore,  integrating \eqref{Dt V2} over $(0,t)$,  and thanks to \eqref{E and V} and \eqref{V and E}, yields that
\begin{equation}
E(t) \leq \biggl(1 + \max\left\{\mu_{1}L, \frac{\mu_{2}}{|\beta|}\right\}\biggr)E(0)e^{-2\lambda t},
\end{equation}
for all $t>0$, which completes the proof.
\end{proof}

\section{Second stability result \textit{via} compactness-uniqueness argument}\label{Sec4}
The second part of this manuscript is devoted to the proof of another stability result of \eqref{sis1}-\eqref{F} stated in Theorem \ref{main2}. To be more precise, we shall show a generic exponential stability result of the solutions to \eqref{sis1}-\eqref{F} by attempting to study the critical set phenomenon of the system.

\subsection{Stability of the linear system}
We first prove that the following observability inequality ensures that the linear system \eqref{2.1} is exponentially stable.
\begin{proposition}\label{OI_prop}
Assume that $\alpha$ and $\beta$ satisfies \eqref{ab} and $L>0$. Thus,  there exists a constant $C>0$, such that for all $\left(u_{0}, z_{0}\right) \in H$
\begin{equation}\label{OI}
\begin{split}
\int_{0}^{L} u_{0}^{2}(x) d x+|\beta| h \int_{0}^{1} z_{0}^{2}(-h \rho) d \rho \mid
\leq C \int_{0}^{T}\left((\partial^2_x u(0, t))^{2}+z^{2}(1, t)\right) d t
\end{split}
\end{equation}
where $(u, z)=S(.)\left(u_{0}, z_{0}(-h \cdot)\right)$ is the solution of the system \eqref{2.1}-\eqref{trans. eq}.
\end{proposition}

Indeed, if \eqref{OI} is true, we get
$$
E(T)-E(0) \leq-\frac{E(0)}{C} \Rightarrow E(T) \leq E(0)-\frac{E(0)}{C} \leq E(0)-\frac{E(T)}{C},
$$
where $E(t)$ is defined by \eqref{energia}. Thus,
\begin{equation}\label{semigroup}
E(T) \leq\gamma E(0), \quad \text{where} \quad  \gamma=\frac{C}{1+C}<1.
\end{equation}
Now,  the same argument used on the interval $[(m-1) T, m T]$ for $m=1,2, \ldots$, yields that
$$
E(m T) \leq \gamma E((m-1) T) \leq \cdots \leq \gamma^{m} E(0).
$$
Thus, we have
$$
E(m T) \leq e^{-\nu m T} E(0) \quad \text{with} \quad
\nu=\frac{1}{T} \ln \left(1+\frac{1}{C}\right)>0.
$$
For an arbitrary positive $t$, there exists $m \in \mathbb{N}^{*}$ such that $(m-1) T<t \leq m T$, and by the non-increasing property of the energy, we conclude that
$$
E(t) \leq E((m-1) T) \leq e^{-\nu(m-1) T} E(0) \leq \frac{1}{\gamma} e^{-\nu t} E(0),
$$
showing the exponential stability result for the linear system.

For sake of clarity, the proof of Proposition \ref{OI_prop} will be achieved by steps.

\vspace{0.1cm}
\noindent\textbf{Step 1: Compactness-uniqueness argument}
\vspace{0.1cm}

We argue by contradiction.  Suppose that \eqref{OI} does not hold and hence there exists a sequence  $\left(\left(u_{0}^{n}, z_{0}^{n}(-h \cdot)\right)\right)_{n} \subset H$ such that
\begin{equation}\label{unit}
\int_{0}^{L}\left(u_{0}^{n}\right)^{2}(x) d x+|\beta| h \int_{0}^{1}\left(z_{0}^{n}\right)^{2}(-h \rho) d \rho=1
\end{equation}
and
\begin{equation}\label{n=0}
\left\| \partial^2_x u^{n}(0, .)\right\|_{L^{2}(0, T)}^{2}+\left\|z^{n}(1, .)\right\|_{L^{2}(0, T)}^{2} \rightarrow 0 \text { as } n \rightarrow+\infty,
\end{equation}
where $\left(u^{n}, z^{n}\right)=S\left(u_{0}^{n}, z_{0}^{n}(-h \cdot)\right)$.

Owing to Proposition \ref{linear}, $\left(u^{n}\right)_{n}$ is a bounded sequence in $L^{2}\left(0, T, H^{2}(0, L)\right)$, and consequently $$\partial_t u^{n}=-\partial_x u^{n}-\partial^3_x u^{n}+\partial_x^5 u \quad \text{is bounded in}\quad  L^{2}\left(0, T, H^{-3}(0, L)\right). $$
Thanks to a result of \cite{Simon},  $\left(u^{n}\right)_{n}$ is relatively compact in $L^{2}\left(0, T, L^{2}(0, L)\right)$ and we may assume that $\left(u^{n}\right)_{n}$ is convergent in $L^{2}\left(0, T, L^{2}(0, L)\right)$. Moreover, using \eqref{2.15} and \eqref{n=0}, we have that $\left(u_{0}^{n}\right)_{n}$ is a Cauchy sequence in $L^{2}(0, L)$.

\vspace{0.1cm}

\noindent\textbf{Claim 1.}\textit{ If $T>h$, then $\left(z_{0}^{n}(-h \cdot)\right)_{n}$ is a Cauchy sequence in $L^{2}(0,1)$.}

\vspace{0.1cm}
In fact, since $z^{n}(\rho, T)=u_{xx}^{n}(0, T-\rho h)$, if $T>h$, we have
$$
\int_{0}^{1}\left(z^{n}(\rho, T)\right)^{2} d \rho =\int_{0}^{1}\left( \partial^2_x u^{n}(0, T-\rho h)\right)^{2} d \rho  \leq \frac{1}{h} \int_{0}^{T}\left(\partial^2_x u^{n}(0, t)\right)^{2} dt.
$$
Using \eqref{2.16}, for $T>h$ yields that
$$
\left\|z_{0}^{n}(-h \cdot)\right\|_{L^{2}(0,1)}^{2}  \leq \frac{1}{h}\left\|\partial^2_x u^{n}(0, \cdot)\right\|_{L^{2}(0, T)}^{2}+\frac{1}{h}\left\|z^{n}(1, \cdot)\right\|_{L^{2}(0, T)}^{2}.
$$
Thus, $\left(z_{0}^{n}(-h \cdot)\right)_{n}$ is a Cauchy sequence in $L^{2}(0,1)$ by means of \eqref{n=0} and hence the Claim 1 is ascertained.

Now, let us pick $\left(u_{0}, z_{0}(-h \cdot)\right)=\lim _{n \rightarrow \infty}\left(u_{0}^{n}, z_{0}^{n}(-h \cdot)\right) \;\; \mbox{in} \; H.$
This, together with \eqref{unit},  yields that
$$\int_{0}^{L} u_{0}^{2}(x) d x+|\beta| h \int_{0}^{1} z_{0}^{2}(-h \rho) d \rho=1.$$
Furthermore, let $(u,z)=S(\cdot)\left(u_{0}, z_{0}(-h \cdot)\right),$
which implies, thanks to Proposition \ref{linear}, that
$$\left(\partial^2_x u (0, \cdot),z(1, \cdot)\right)=\lim _{n \rightarrow \infty}\left(\partial^2_x u^{n}(0, \cdot), z^{n}(1, \cdot)\right)$$ in $L^{2}(0, T)$.
Combining the latter with \eqref{n=0} gives
$\left(\partial^2_x u(0, \cdot), z(1, .)\right)=0.$ As we have $z(1, t)=\partial^2_x u(0, t-h)=0$, we deduce
that $z_{0}=0$ and $z=0$. Consequently $u$ is solution of
\begin{equation}\label{kk}
\begin{cases}
\partial_t u-\partial_x u+\partial^3_x u-\partial^5_x u=0, & x \in(0, L), t>0 \\
u(0, t)=u(L, t)=\partial_x u(L, t)=\partial_x u(0, t)=\partial^2_x u(L, t)=\partial^2_x u(0, t)=0, & t>0 \\
u(x, 0)=u_{0}(x), & x \in(0, L)
\end{cases}
\end{equation}
with
\begin{equation}\label{cont}
\left\|u_{0}\right\|_{L^{2}(0, L)}=1 .
\end{equation}

\noindent\textbf{Step 2: Reduction to a spectral problem}
\begin{lemma}\label{SP} For any $T>0$, let $N_{T}$ denote the space of the initial state $u_{0}\in L^2(0,L)$,  such that the solution of the Kawahara system $u(t)=S(t)u_{0}$ satisfies \eqref{kk}. Then, $N_{T}=\{0\}$.
\end{lemma}
\begin{proof} We argue as in \cite[Theorem 3.7]{ro}.  If $N_{T} \neq\{0\}$, then the map $ u_{0}\in \mathbb{C} N_{T} \rightarrow A\left(N_{T}\right) \subset \mathbb{C} N_{T}$ ($\mathbb{C} N_{T}$ denotes the complexification of $N_{T}$) has (at least) one eigenvalue. Hence, there exists a pair $(\lambda,u_0) \in \mathbb{C}\times H^{5}(0, L) \backslash\{0\}$ such that
$$
\begin{cases}\lambda u_{0}+u_{0}^{\prime}+ u_{0}^{\prime \prime \prime}- u_{0}^{\prime \prime \prime \prime \prime}=0, & \text { in }(0, L), \\
 u_{0}(0)=u_{0}(L)=u_{0}^{\prime}(0)=u_{0}^{\prime}(L)=u_{0}^{\prime \prime}(0)=u_{0}^{\prime \prime}(L)=0 . & \end{cases}
$$
To obtain the contradiction, it remains to prove that such a pair $\left(\lambda, u_{0}\right)$ does not exist. This will be done in the next step.
\end{proof}

\noindent\textbf{Step 3: M\"obius transformation}
\vglue 0.2cm
To simplify the notation, henceforth we denote $u_0:=u$.  Moreover, the notation
$\{0, L\}$ means that the function is applied to $0$ and $L$, respectively.
\begin{lemma}\label{lem2}
Let $L>0$ and consider the assertion
\begin{equation*}
(\NN):\ \ \exists \lambda \in \C,  \exists u  \in H^2_0(0,L)\cap H^5(0,L) \,\, \text{such that}\
\begin{cases}
\lambda u +u'+u'''-u'''''=0, & \text{on} \,\, (0,L),  \\
u(x)=u'(x)=u''(x)=0, & \text{in} \,\, \{0,L\}.
\end{cases}
\end{equation*}
If $(\lambda,u) \in \C \times  H^2_0(0,L)\cap H^5(0,L)$ is solution of $(\NN)$, then
$u=0.$
\end{lemma}
\begin{proof}
Consider the following system
\begin{equation}\label{e8}
\begin{cases}
\lambda u+u'+u'''-u'''''=0, & \text{on} \,\, (0,L),  \\
u(x)=u'(x)=u''(x)=0, & \text{in} \,\, \{0,L\}.
\end{cases}
\end{equation}
Multiplying the equation \eqref{e8} by $\overline{u}$ and integrating in $[0,L]$,  we have that $\lambda$ is purely imaginary, i.e., $\lambda=ir$, for $r\in\mathbb{R}$. Now,  extending the function $u$ to $\mathbb{R}$ by setting  $u=0$ for $x\not\in [0,L]$, we have that the extended function satisfies
 \begin{equation*}
 \lambda u+ u'+u'''-u'''''=-u''''(0)\delta_0^{'}+u''''(L)\delta_L^{'}-u'''(0)\delta_0+u'''(L)\delta_L,
\end{equation*}
in ${\mathcal S }'(\R )$,  where $\delta_\zeta$ denotes the Dirac measure at $x=\zeta$ and the derivatives $u''''(0)$, $u''''(L)$, $u'''(0)$ and $u'''(L)$ are those of the
function $u$ when restricted to $[0,L]$.  Taking the Fourier transform of each term in the above system and integrating by parts, we obtain
\begin{equation*}
\begin{split}
\lambda \hat u (\xi )  + i\xi \hat u(\xi )  +(i\xi )^3 \hat u (\xi) -  (i\xi )^5 \hat u (\xi)  = -(i\xi)u'''(0)+(i\xi)u'''(L)e^{-iL\xi}-u''''(0)+u''''(L)e^{-iL\xi}.
\end{split}
\end{equation*}
Take $\lambda = -ir$ and let $f_{\alpha}(\xi,L)=i \hat u (\xi )$. The latter gives
\begin{equation*}
f_{\alpha}(\xi,L)=\frac{N_{\alpha}(\xi,L)}{q(\xi)},
\end{equation*}
where $N_{\alpha}(\cdot,L)$ is defined by
\begin{equation}\label{Na}
N_{\alpha}(\xi,L)=\alpha_1i\xi-\alpha_2i\xi e^{-i\xi L}+\alpha_3-\alpha_4e^{-i\xi L}
\end{equation}
and
\begin{equation*}
q(\xi)=\xi^5+\xi^3-\xi+r,
\end{equation*}
where
$\alpha_i$, for $i=1,2,3,4$, are the traces of $u'''$ and $u''''$.

For each $r\in\mathbb{R}$ and $\alpha\in\mathbb{C}^4\setminus \{0\}$, let $\FF_{\alpha r}$ be the set of $L>0$ values, for which the function $f_{\alpha}(\cdot,L)$ is entire. 
Now, let us recall the equivalent following statements:
\begin{itemize}
\item[A1.] $f_{\alpha}(\cdot,L)$ is entire;
\item[A2.] all zeros, taking the respective multiplicities into account, of the polynomial $q$ are zeros of $N_{\alpha}(\cdot,L)$;
\item[A3.]  the maximal domain of $f_{\alpha}(\cdot,L)$ is $\mathbb{C}$.
\end{itemize}
Whereupon, the function $f_{\alpha}(\cdot,L)$ is entire, due to the equivalence between statement A1 and A2, if the following holds
\begin{equation*}
\frac{\alpha_1i\xi_i+\alpha_3}{\alpha_2i\xi_i+\alpha_4}=e^{-iL\xi_i},
\end{equation*}
where $\xi_i$ denotes the zeros of $q(\xi)$, for $i=1,2,3,4,5$.
Thereafter, let us define, for $\alpha \in\mathbb{C}^4\setminus \{0\}$, the following discriminant
\begin{equation}\label{discri}
d(\alpha)=\alpha_1\alpha_3-\alpha_2\alpha_4.
\end{equation}
Then, for $\alpha\in\mathbb{C}^4\setminus \{0\}$, such that $d(\alpha)\neq0$ the M\"obius transformations can be introduced by
\begin{equation}\label{mobius}
M(\xi_i)=e^{-iL\xi_i},
\end{equation}
for each zero $\xi_i$ of the polynomial $q(\xi)$.

\vspace{0.1cm}

\noindent The next claim describes the behavior of the roots of polynomial $q(\cdot)$:

\vspace{0.1cm}

\noindent\textbf{Claim 2.}
\textit{The polynomial $q(\cdot)$ has exactly one real root with multiplicity $1$ and two pairs of complex conjugate roots.}
\vspace{0.1cm}

Indeed, we suppose that $r\neq 0$ (the case $r=0$ will be discussed later). Note that the derivative of $q$ is given by
\begin{equation*}
q'(\xi)=5\xi^4+3\xi^2-1,
\end{equation*}
and its zeros are $\pm z_1$ and $\pm z_2$, where
\begin{equation*}
z_1=\sqrt{\frac{-3-\sqrt{29}}{10}} \quad\text{and}\quad z_2=\sqrt{\frac{-3+\sqrt{29}}{10}}.
\end{equation*}
It is easy to see that $z_1$ belongs to $\C \setminus \R$ and $z_2$ belongs to $\R$. Hence, the polynomial $q(\cdot)$ does not have critical points, which means that $q(\cdot)$ has exactly one real root. Suppose that $\xi_0 \in \R$ is the root of $q(\cdot)$ with multiplicity $m \leq 5$. Consequently,
$$q(\xi_0)= q'(\xi_0)= ... = q^{(m-1)}(\xi_0)=0.$$

Consider the following cases:
\begin{enumerate}
\item[(i)] If $\xi_0$ has multiplicity $5$, it follows that $q(\xi_0)=0$ and $q''''(\xi_0)=-120\xi_0=0$,  implying that $\xi_0=0$ and $r=0$.
\item[(ii)] If $\xi_0$ has multiplicity $4$, it follows that $q'''(\xi_0)=60\xi^2_0+6=0$ and thus $\xi_0 \in i\R$.
\item[(iii)] If $\xi_0$ has multiplicity $3$, it follows that $q(\xi_0)=0$ and $q''(\xi_0)=20\xi^3_0+6\xi_0=0$ and hence $\xi_0=0$ and $r=0$ or $\xi_0 \in i\R$.
\item[(iv)] If $\xi_0$ has multiplicity $2$, it follows that $q'(\xi_0)=5\xi^4_0+3\xi_2-1=0$,  implying that $\xi_0 \in \C \setminus \R.$
\end{enumerate}
Note that in all cases, since $r\neq 0$ and $\xi_0 \in \R$, we get a contradiction.  Consequently, $q(\cdot)$ has exactly one real root, with multiplicity $1$. This means that this polynomial has two pairs of complex conjugate roots.

Now, we assume that $r=0$. Then, we obtain that $ q(\xi)= \xi(\xi^4+\xi^2-1),$ whose roots are $0, \pm   \rho$ and $\pm  k$ where
\begin{equation}
\rho=\sqrt{\frac{\sqrt5-1}{2}} \quad \text{and} \quad k=i\sqrt{\frac{1+\sqrt5}{2}}
\end{equation}
Thus, $q(\cdot)$ has two pairs of complex conjugate roots and three real roots, proving Claim 2.

Further to Claim 2, and in order to conclude the proof of Lemma \ref{lem2}, we need two additional lemmas whose proofs are given in \cite{santos} (see Lemmas 2.1 and 2.2).

\begin{lemma}\label{mobiuslemma2}
Let non null $\alpha \in   \C^4$ with $d(\alpha) = 0$ and $L > 0$ for $d(\alpha)$ defined in \eqref{discri}. Then, the set of the imaginary parts of the
zeros of $N_{\alpha}(\cdot,L)$ in \eqref{Na} has at most two elements.
\end{lemma}

\begin{lemma}\label{mobiuslemma}
For any $L > 0$, there is no M\"obius transformation $M$, such that
\begin{align*}
M(\xi)=e^{-iL\xi}, \quad \xi\in \{ \xi_1, \xi_2, \bar{\xi}_1, \bar{\xi}_2 \},
\end{align*}
with $\xi_1, \xi_2, \bar{\xi}_1, \bar{\xi}_2$ all distinct in $\C$.
\end{lemma}

We are now in position to prove the Lemma \ref{lem2}.  Let us consider two cases: \begin{itemize}
\item[i.] $d(\alpha)\neq 0$;
\item[ii.] $d(\alpha)=0$,
\end{itemize}
where $d(\alpha)$ was defined in \eqref{discri}.

First,  supposing that $d(\alpha)\neq 0$,  we are able to define the M\"obius transformation. In fact,  suppose by contradiction that there exists $L>0$ such that the function $f_a(\cdot,L)$ is entire. Then, all roots of the polynomial $q(\cdot)$ must satisfy \eqref{mobius}, i.e., there exists a M\"obius transformation that takes each root $\xi_0$ of $q(\cdot)$ into $e^{-iL\xi_0}$. However, this contradicts Lemma \ref{mobiuslemma} and proves that if $(\NN)$ holds then $\FF_{\alpha r}=\emptyset$ for all $r\in\R$. On the other hand, suppose that $d(\alpha)=0$ and note that by using the claim 2, we can conclude that the set of the imaginary parts of the polynomial $q(\cdot)$ has at least three elements, thus it follows from Lemma \ref{mobiuslemma2} that $\FF_{\alpha r}=\emptyset$ for all $r\in\R$. Note that in both cases, we have that $\FF_{\alpha r}=\emptyset$, which implies that $(\NN)$ only has the trivial solution for any $L>0$, and the proof of Lemma \ref{lem2} is archived.
\end{proof}
\begin{proof}[Proof of Proposition \ref{OI_prop}]
Notice that \eqref{cont} implies that the solution $u$ can not be identically zero. However, from  Lemma \ref{SP}, one can conclude that $u=0$, which drives us to a contradiction.
\end{proof}

\subsection{Proof of Theorem \ref{main2}}
Let us consider the nonlinear Kawahara system \eqref{sis1}-\eqref{F}. Consider $
\left\|\left(u_{0}, z_{0}\right)\right\|_H \leq r,$
where $r$ will be chosen later.  The solution $u$ of \eqref{sis1}-\eqref{F}, with $p=2$, can be written as $u=u_1+u_2$, where $u_1$ is the solution of
$$\begin{cases}\partial_t u_{1}-\partial^5_x u_1+\partial^3_x u_{1}+\partial_x u_{1}=0, & x \in(0, L), t>0 \\ u^{1}(0, t)=u^{1}(L, t)=\partial_x u_1(0,t)=\partial_x u_1(L,t)=0, & t>0 \\ \partial^2_x u_{1}(L, t)=\alpha \partial^2_x u_{1}(0, t)+\beta \partial^2_x u_{1}(0, t-h), & t>0 \\ \partial^2_x u_{1}(0, t)=z_{0}(t), & t \in(-h, 0) \\ u^{1}(x, 0)=u_{0}(x), & x \in(0, L),\end{cases}$$
and $u_{2}$ is solution of
$$\begin{cases}
\partial_t u_{2}-\partial^5_x u_2+\partial^2_x u_{2}+\partial_x u_{2}=-u^2 \partial_x u, & x \in(0, L), t>0 \\ u_{2}(0, t)=u_{2}(L, \\ \partial^2_x u_{2}(L, t)=\alpha \partial^2_x u_{2}(0, t)+\beta \partial^2_x u_{2}(0, t-h), & t \in(-h, 0) \\ \partial^2_x u_{2}(0, t)=0, & x \in(0, L),  \\ u_{2}(x, 0)=0, & x \in(0,L),\end{cases}$$
Note that, in this case,  $u_1$ is the solution of \eqref{2.1}-\eqref{trans. eq} with the initial data $\left(u_{0}, z_{0}\right)\in H$ and $u_2$ is solution of \eqref{with source term} with null data and right-hand side $f=u^2 \partial_x u \in L^1(0,T;L^2(0,L))$,  as in Lemma \ref{uux}.

Now, thanks to \eqref{semigroup},  Proposition \ref{source term} and Lemma \ref{uux}, we have that
\begin{equation}\label{31}
\begin{split}
\|(u(T), z(T))\|_{H} \leq&\left\|\left(u^{1}(T), z^{1}(T)\right)\right\|_{H}+\left\|\left(u^{2}(T), z^{2}(T)\right)\right\|_{H}\\
\leq& \gamma\left\|\left(u_{0}, z_{0}(-h \cdot)\right)\right\|_{H}+C\left\|u^p u_{x}\right\|_{L^{1}\left(0, T, L^{2}(0, L)\right)}\\
\leq& \gamma\left\|\left(u_{0}, z_{0}(-h \cdot)\right)\right\|_{H}+C\|u\|_{L^{2}\left(0, T, H^{2}(0, L)\right)}^{2},
\end{split}
\end{equation}
with $\gamma\in(0,1)$. The goal now is to deal with the lest term of the previous inequality. To this end, we use the multipliers method.  First, we multiply the first equation of \eqref{sis1}-\eqref{F} by $xu$ and integrate by parts to obtain
\begin{equation*}
\begin{split}
\frac{1}{2} \int_{0}^{L} x|u(x, T)|^{2} d x+\frac{3}{2} \int_{0}^{T} \int_{0}^{L}\left| \partial_x u(x, t)\right|^{2} d x d t+\frac{5}{2} \int_{0}^{T} \int_{0}^{L}\left| \partial^2_xu(x, t)\right|^{2} d x dt \\
=\frac{1}{2} \int_{0}^{T} \int_{0}^{L}|u(x, t)|^{2} d x dt+\frac{L}{2}\int_0^T (\partial^2_x u(L,t))^2 dt+\frac{1}{2} \int_{0}^{L} x\left| u_{0}(x)\right|^{2} d x +\frac{1}{4} \int_{0}^{T} \int_{0}^{L}\left|u\right|^{4} dxdt.
\end{split}
\end{equation*}
Consequently,  using the boundary condition of \eqref{sis1}-\eqref{F} and \eqref{energia}, we get
\begin{equation*}
\begin{split}
 \int_{0}^{T} \int_{0}^{L}\left| \partial_x u (x, t)\right|^{2} d x d t&+ \int_{0}^{T} \int_{0}^{L}\left|\partial^2_x u(x, t)\right|^{2} d x d t\leq(T+L)\left\|\left(u_{0}, z_{0}\right)\right\|_{H}^2\\&+ L\int_0^T(\alpha \partial^2_x u (0,t)+\beta z(1,t))^2dt+\frac{1}{2}\int_0^T\int_0^L\left|u\right|^{4} dxdt.
\end{split}
\end{equation*}
Note that Gagliardo–Nirenberg inequality ensures that
\begin{equation*}
	\begin{split}
		\int_{0}^{T} \int_{0}^{L} u^{4} d x d t  \leq &C \int_{0}^{T}\|u\|_{L^2(0,L)}^{3}\left\|u_{x}\right\|_{L^2(0,L)} d t	\\
		\leq & C\frac{1}{2\varepsilon}\int_{0}^{T}\|u\|_{L^2(0,L)}^{6} d t+C \frac{\varepsilon}{2} \int_{0}^{T}\left\|u_{x}\right\|_{L^2(0,L)}^{2} d t\\
		\leq & C(T)\frac{1}{2\varepsilon}\|u\|_{L^{\infty}\left(0, T ; L^{2}(0, L)\right)}^{6}+C \frac{\varepsilon}{2}\|u\|_{L^{2}\left(0, T ; H^{2}(0, L)\right)}^{2}\\
		\leq & C(T)\frac{1}{2\varepsilon}\left\|(u_{0},z_0)\right\|^{6}_{H}+C \frac{\varepsilon}{2}\|u\|_{L^{2}\left(0, T ; H^{2}(0, L)\right)}^{2}.
	\end{split}
\end{equation*}
Putting together the previous inequalities we have
\begin{equation}\label{32}
\begin{split}
 &\int_{0}^{T} \int_{0}^{L}\left|\partial_x u(x, t)\right|^{2} dx dt+ \int_{0}^{T} \int_{0}^{L}\left| \partial^2_x u (x, t)\right|^{2} dx dt\leq(T+L)\left\|\left(u_{0}, z_{0}\right)\right\|_{H}^2\\&+ L\int_0^T(\alpha \partial^2_x u(0,t)+\beta z(1,t))^2dt+C(T)\frac{1}{2\varepsilon}\left\|(u_{0},z_0)\right\|^{6}_{H}+C \frac{\varepsilon}{2}\|u\|_{L^{2}\left(0, T ; H^{2}(0, L)\right)}^{2}.
\end{split}
\end{equation}

Now, multiplying the first equation of \eqref{sis1} by $u$ and integrating by parts yields that
$$
\int_{0}^{L} u^{2}(x, T) d x-\int_{0}^{L} u_{0}^{2}(x) d x -\int_{0}^{T}\left(\alpha u_{xx}(0, t)+\beta z(1, t)\right)^{2} d t+\int_{0}^{T} u_{x}^{2}(0, t) dt=0
$$
Using the same idea as in the proof of \eqref{2.14},  we have that
$$
\int_{0}^{T} (\partial^2_x u)^{2}(0, t) d t+\int_{0}^{T} z^{2}(1, t) d t \leq C\left\|\left(u_{0}, z_{0}\right)\right\|_{H}^{2}.
$$
Consequently,  the previous inequality gives
\begin{equation*}
\begin{split}
\int_{0}^{T}\left(\alpha \partial^2_x u(0, t)+\beta z(1, t)\right)^{2} dt
\leq 2 C\left(\alpha^{2}+\beta^{2}\right)\left\|\left(u_{0}, z_{0}\right)\right\|_{H}^{2}.
\end{split}
\end{equation*}
Thus,  putting the previous inequality in \eqref{32},  and choosing $\varepsilon>0$ sufficiently small,  there exists $C>0$ such that
\begin{equation}\label{33}
\|u\|_{L^{2}\left(0, T ; H^{2}(0, L)\right)}^{2}\leq C\left(\left\|\left(u_{0}, z_{0}\right)\right\|_{H}^{2}+\left\|\left(u_{0}, z_{0}\right)\right\|_{H}^{6}\right).
\end{equation}
Finally,  gathering \eqref{31} and \eqref{33}, there exists $C>0$ such that the following holds true
$$
\|(u(T), z(T))\|_{H} \leq\left\|\left(u_{0}, z_{0}\right)\right\|_{H}\left(\gamma+C\left\|\left(u_{0}, z_{0}\right)\right\|_{H}+C\left\|\left(u_{0}, z_{0}\right)\right\|_{H}^{5}\right),
$$
which implies
$$
\|(u(T), z(T))\|_{H} \leq\left\|\left(u_{0}, z_{0}\right)\right\|_{H}\left(\gamma+C r+C r^{5}\right) .
$$
Given $\epsilon>0$ small enough such that $\gamma+\epsilon<1$, we can take $r$ small enough such that $r+r^{5}<\frac{\epsilon}{C}$, in order to have
$$
\|(u(T), z(T))\|_{H} \leq(\gamma+\epsilon)\left\|\left(u_{0}, z_{0}\right)\right\|_{H},
$$
with  $\gamma+\epsilon<1$.  Theorem \ref{main2} follows using the semigroup property as in \eqref{semigroup}.  \qed

\section{Further comments and open problems} \label{Sec5}
Our work presents a further step after the work \cite{ara} for a better understanding of the stabilization problem for the Kawahara equation. Indeed, a boundary time-delayed damping control is proposed to stabilize the equation in contrast to \cite{ara}, where an interior damping is required and no delay is taken into consideration.  We conclude our paper with a few comments and also some open problems.
\begin{remark} In what concerns our main results, Theorems \ref{Lyapunov} and \ref{main2}, the following remarks are worth mentioning:
\begin{itemize}
\item Note that the rate $\lambda$ of the Theorem \ref{Lyapunov} decreases as the delay $h$ increases, since we have the restriction \eqref{lambda}.
\item A simple calculation shows that taking $\mu_{1}, \mu_{2} \in (0,1)$ in Theorem \ref{Lyapunov} such that
	\begin{equation*}
		\mu_{2} < \min\left\{ 1- |\beta|-\alpha^{2}, \frac{(|\beta| - 1)^{2}- \alpha^{2}}{1 - |\beta|}, \frac{\alpha^{2} - \beta^{2} + |\beta|}{|\beta|}\right\}
	\end{equation*}
	and
	\begin{equation*}
		\mu_{1} < \min\left\{ \frac{1 -|\beta|- \mu_{2} - \alpha^{2} }{L \alpha^{2} }, \frac{(|\beta| - 1)^{2} - \alpha^{2} - \mu_{2}(1 - |\beta|)}{L(\alpha^{2} - \beta^{2} + |\beta|(1 - \mu_{2}))} \right\}
	\end{equation*}
implies that $M_{\mu_{1}}^{\mu_{2}}$ is negative definite.
\item Note that the presence of the nonlinearity on the equation yields the restriction about the initial data. Hence, if we remove it, that is, by considering the linear system, it is possible to obtain the same result of the Theorem \ref{Lyapunov}, with the same process. Nevertheless, the decay rate $\lambda$ is given by
	\begin{equation}
		\lambda \leq \min\left\{ \frac{\mu_{2}}{2h(\mu_{2} + |\beta|)} ,\frac{3b\pi^2-L^2a}{2L^2(1+L\mu_1)} \right\}.
	\end{equation}
	\item For sake of simplicity, we only considered in this article the nonlinearity $u^2u_x$.  However, Theorems \ref{Lyapunov} and \ref{main2} are still valid for $u^pu_x$,  $p\in[1,2)$, where the proof is very similar and hence omitted.

	\item Recently,  Zhou \cite{Zhou} proved the well-posedness of the following initial boundary value problem
\begin{equation}\label{int2}
\left\{\begin{array}{lll}
\partial_{t} u-\partial_x^5 u=c_{1} u \partial_x uu +c_{2} u^{2} \partial_x u+b_{1}  \partial_x u \partial_x^2 u+b_{2} u \partial_x^3 u,& x \in(0, L), \ t \in \mathbb{R}^{+}, \\
u(t, 0)=h_{1}(t), \quad u(t, L)=h_{2}(t),  \quad \partial_x u (t, 0)=h_{3}(t), &  t \in \mathbb{R}^{+},\\
\partial_x u (t, L)=h_{4}(t), \quad \partial_x^2 u(t, L)=h(t), &  t \in \mathbb{R}^{+},\\
u(0, x)=u_{0}(x),& x \in(0, L),
\end{array}\right.
\end{equation}
Thus, due to this result,  when we consider $b_1=b_2=0$ and the combination $c_1u \partial_x u+c_2u^2 \partial_x u$ instead of $u^p \partial_x u$, for $p\in[1,2]$, in \eqref{sis1}, the main results of our article remains valid.
\item We point out that considering $a=0$ in \eqref{sis1},  Theorem \ref{Lyapunov} holds true. Additionally,  no restriction is necessary in the length $L>0$, and also Theorem \ref{main2} is still verified (see, for instance, \cite{cajesus,vasi}).
\end{itemize}
\end{remark}

\subsection{Open problems} Based on the outcomes of this paper on the dispersive Kawahara equation, some interesting open problems appear.

\subsubsection{Restriction in the Lyapunov approach}
Observe that in our first result, Theorem \ref{Lyapunov},  due the fact that the result is based on the appropriate choice of Lyapunov functional,  we have a restriction \eqref{L} on the length $L$. This is due to the choice of the Morawetz multipliers $x$ in the expression of $V_1$ defined by \eqref{V1}.  Therefore, the following natural question arises.

\vspace{0.1cm}
 \noindent\textbf{Question $A$:} Can we choose another Lyapunov functional, instead of the previous one to remove the restriction over $L$?

\subsubsection{Critical set} As observed in \cite{ara}, considering the following initial boundary value problem for Kawahara equation
\begin{equation}\label{kka}
\begin{cases}
u_{t}-u_x+u_{xxx}-u_{xxxxx}=0, & x \in(0, L), t>0 \\
u(0, t)=u(L, t)=u_{x}(L, t)=u_{x}(0, t)=u_{xx}(L, t)=0, & t>0, \\
u(x, 0)=u_{0}(x), & x \in(0, L),
\end{cases}
\end{equation}
it is possible to construct a nontrivial steady-state solution to  \eqref{kka} with a non-zero initial datum $u_{0}(x)\not \equiv 0$ and homogeneous boundary conditions upon the endpoints of the interval with a critical length. Precisely,  when the authors considered the following constants
\[
a=\sqrt{{\sqrt{5}+1}/2},\ \ b=\sqrt{{\sqrt{5}-1}/2},  \ \ A=C_{2}+C_{3},\ \ B=C_{2}-C_{3}
\]
\[
 C_{2}=1-e^{-aL}, \ \ C_{3}=e^{aL}-1, \ \ C_{1}=-\left(  1+\frac{a^{2}}{b^{2}}\right)  A,\ \ C_{4}=\frac{a^{2}}{b^{2}}A,\ \ C_{5}=-\frac{a}{b}B,
\]
they were able to define the set
\[
\mathcal{N}=\left\{  L>0:e^{ibL}=\left(  \frac{C_{4}+iC_{5}}{|C_{4}+iC_{5}%
|}\right)  ^{2}\right\}  \subset\mathbb{R}^{+}%
\]
and
\[
u(x)=C_{1}+C_{2}e^{ax}+C_{3}e^{-ax}+C_{4}\cos(bx)+C_{5}\sin(bx)\not \equiv
0,\ \ x\in(0,L).
\]
If $L\in\mathcal{N},$ then $u=u(x)$ solves
$-u^{\prime\prime\prime\prime\prime}+u^{\prime\prime\prime}+u^{\prime}=0,$
and satisfies $u(0)=u^{\prime}(0)=u^{\prime\prime}(0)=u(L)=u^{\prime}(L)=u^{\prime\prime
}(L)=0.$

So, in our context, if we consider a function $N_{\alpha}:\mathbb{C}\times(0,\infty)\to\mathbb{C},$ with $\alpha\in\mathbb{C}^4\setminus \{0\}$, whose restriction $N_{\alpha}(\cdot,L)$, given by \eqref{N},  is entire for each $L>0$ and a family of functions $f_{\alpha}(\cdot,L)$, defined by \eqref{f},  in its maximal domain, the following issue appears.

\vspace{0.1cm}
\noindent\textbf{Question ${B}$:}  Is it possible to find $a\in\mathbb{C}^4\setminus \{0\}$ such that the function $f_a(\cdot,L)$ is an entire function?

\vspace{0.1cm}

Note that the proof of Theorem \ref{main2} heavily relies on a unique continuation property of the spectral problem associated with the space operator (see Lemma \ref{lem2}). However, due the structure of the terms  $\partial_x^3$ and $\partial_x^5$ (see Lemma \ref{lem2}), we are unable to study the spectral problem in a direct way as in \cite{ro}. Hence, due of these two different dispersions third and fifth order, we believe that a new approach is needed to tackle the previous open question.

\subsection*{Acknowledgments} Capistrano–Filho was supported by CNPq grant 307808/2021-1,  CAPES grants 88881.311964/2018-01 and 88881.520205/2020-01,  MATHAMSUD grant 21-MATH-03 and Propesqi (UFPE).   De Sousa acknowledges support from CAPES-Brazil and CNPq-Brazil.  Gonzalez Martinez was supported by FACEPE grant BFP-0065-1.01/21.  This work is part of the PhD thesis of De Sousa at Department of Mathematics of the Universidade Federal de Pernambuco.


\end{document}